\documentclass[11pt, a4paper,pdftex]{amsart} 
\usepackage[foot]{amsaddr}

\usepackage{mathpazo}

\usepackage[mathcal]{eucal}
\DeclareFontFamily{OT1}{pzc}{}
\DeclareFontShape{OT1}{pzc}{m}{it}{<-> s * [1.10] pzcmi7t}{}
\DeclareMathAlphabet{\mathpzc}{OT1}{pzc}{m}{it}

\usepackage[protrusion=true,expansion=true]{microtype}

\usepackage{hyperref}
\usepackage{amssymb}
\usepackage{amsmath}
\usepackage{mathrsfs}  
\usepackage{latexsym}
\usepackage{color}
\usepackage{enumitem}
\usepackage{graphicx}
\usepackage{blkarray}
\usepackage{bbm} 

\usepackage{mathtools}

\usepackage{tikz-cd}

\usepackage{todonotes}

\usepackage[vmargin=3cm, hmargin=3cm]{geometry}
\numberwithin{equation}{subsection}

\newtheorem{thmA}{Theorem}

\newtheorem{theorem}{Theorem}[subsection]  
\newtheorem{lemma}[theorem]{Lemma} 
\newtheorem{proposition}[theorem]{Proposition}
\newtheorem{corollary}[theorem]{Corollary}

\theoremstyle{remark} 
\newtheorem{definition}[theorem]{Definition}
\newtheorem{remark}[theorem]{Remark}

\newcommand{\Inj}{\mathscr{I}}
\newcommand{\Operad}{\mathscr{E}}

\newcommand{\Mac}{\mathcal{M}}

\DeclareMathOperator*{\argmin}{arg~min} 

\title{An $E_\infty$ structure on the matroid grassmannian}

\author{Jeffrey Giansiracusa}
\email{jeffrey.giansiracusa@durham.ac.uk} 

\date{\today}

\begin{document}
\begin{abstract}
In analogy with the origin of the additive structure of $K$-theory, we construct an $E_\infty$
structure on the matroid Grassmannian (the space of oriented matroids) for which the underlying
binary operation is the direct sum of matroids.  The proof involves lifting the polyhedral fan
structure of the Dressian to a polyhedral model for the matroid Grassmannian, and introducing a
novel $E_\infty$ operad made from the space of infinite subsets of $\mathbb{N}$.
\end{abstract}
\maketitle

In \cite{Gelfand-MacPherson}, Gelfand and MacPherson found a long-sought combinatorial formula for
the rational Pontrjagin classes of a triangulated manifold. Their formula made essential conceptual
use of oriented matroids \cite{Oriented-matroids-book}. Motivated by this breakthrough, MacPherson
\cite{MacPherson-CD-manifolds} proposed studying a combinatorial analogue of the category of smooth
manifolds that he called \emph{combinatorial differential manifolds} or CD-manifolds.  In this
category, the tangent bundle theory is controlled by the \emph{matroid Grassmannian}, which is the
nerve of the poset of oriented matroids of fixed rank and on a fixed ground set. For this reason,
the matroid Grassmannian is sometimes denoted $\mathit{MacP}(d,E)$ ($d$ being the rank and $E$ being
the ground set). MacPherson observed that when $d=1,2, |E|-1$, or $|E|-2$, the matroid Grassmannian
is homeomorphic to the familiar Grassmannian $\mathit{Gr}(d,E)$ of $d$-dimensional subspaces of
$\mathbb{R}^E$; outside of these cases, the homotopy type was a mystery, and understanding it would
have important implications for the theory of characteristic classes.

A decade later, Biss claimed a proof that the matroid Grassmannian is homotopy equivalent to the
real Grassmannian in all cases \cite{Biss-Annals}.  Unfortunately, the argument contained a fatal
flaw and the paper was retracted \cite{Biss-Annals-retraction, Mnev-error}.  A related claim
\cite{Biss-complex} of a combinatorial matroid-based model for the complex Grassmannian suffered
from the same mistake and was also withdrawn \cite{Biss-complex-erratum}. Since then, the problem of
understanding the homotopy type of the matroid Grassmannian remains largely open, although Anderson
and Davis have proved various results illuminating this question and the theory of CD-manifolds in
\cite{Anderson-matroid-bundles,Anderson-homotopy,Anderson-mod2,Anderson-hyperfields}.  In
particular, they showed that the comparison map
\[
	\mathit{Gr}(d,E) \to \mathit{MacP}(d,E)
\]
is a split surjection on mod 2 cohomology.

The purpose of this paper prove a theorem saying that at least one important part of the homotopical
structure of the real Grassmannains is also carried by the matroid Grassmannians: the $E_\infty$
structure corresponding to direct sum.  Direct sum induces maps
\[
	\mathit{Gr}(d_1,E_1) \times \mathit{Gr}(d_2,E_2) \to \mathit{Gr}(d_1 + d_2, E_1 \sqcup E_2)
\]
that yield an $E_\infty$ structure on $\coprod^\infty_{d=0} \mathit{Gr}(d,\mathbb{N})$.  This structure
underlies the additive structure of real $K$-theory.  

\begin{thmA}
The disjoint union $\coprod_{d=0}^\infty \mathit{MacP}(d,\mathbb{N})$ carries an action of an
$E_\infty$ operad extending the direct sum of oriented matroids, and hence its group completion has
the homotopy type of an infinite loop space.
\end{thmA}

Our method is to replace the matroid Grassmannian $\mathit{MacP}(d,E)$ with a more geometric cell
complex $\Mac(d,E)$ living inside the projective space  $\mathbb{P}(\mathbb{R}^{\binom{|E|}{d}})$ and on
which we are able to give an explicit and simple formula for the operad action. The complex
$\Mac(d,E)$ is the analogue for oriented matroids of the Dressian that parametrizes valuated
matroids in tropical geometry, and it can be realized as the Grassmannian over the hyperfield of
signed tropical numbers in the sense of \cite{Baker-Bowler-partial-hyperstructures}. Anderson and
Davis proved in \cite{Anderson-hyperfields} that $\Mac(d,E)$ and $\mathit{MacP}(d,E)$ are weakly
equivalent when the former is given a topology induced by a topology on $\mathbb{R}$ slightly
different from the usual one (motivated by the signed tropical hyperfield structure, the 0-coarse
topology agrees with the usual topology away from 0, but the only neighbourhood of 0 is the whole
line).

We instead work with the Euclidean-induced topology on $\Mac(d,E)$ and show that it is a polyhedral
complex when viewed in logarithmic coordinates. More precisely, the group homomorphism
$(\mathbb{R},\times) \to \mathbb{T} = (\mathbb{R} \cup \{\infty\}, +)$ given by $x\mapsto -\log |x|$
induces a map of projective spaces
\[
\mathbb{P}(\mathbb{R}^{\binom{|E|}{d}}) \to \mathbb{P}(\mathbb{T}^{\binom{|E|}{d}}),
\]
where the object on the right, often referred to as tropical projective space, is a polyhedral
compactification of the Euclidean space $\mathbb{R}^{\binom{|E|}{d}-1}$ that is combinatorially
equivalent to a simplex.  We prove:
\begin{thmA}
The space $\Mac(d,E) \subset \mathbb{P}(\mathbb{R}^{\binom{|E|}{d}} )$ is a CW complex such
that the map $\mathbb{P}(\mathbb{R}^{\binom{|E|}{d}}) \to \mathbb{P}(\mathbb{T}^{\binom{|E|}{d}})$ 
sends each cell homeomorphically onto a convex polyhedron.
\end{thmA}

\begin{thmA}
The $\Mac(d,E)$ (with the subspace topology) is homotopy equivalent
to the matroid Grassmannian $\mathit{MacP}(d,E)$.
\end{thmA}

The proof of the above homotopy equivalence involves giving an combinatorial description of the
structure of the poset of cells of $\Mac(d,E)$ and then comparing this poset with the poset of oriented matroids.

\subsection*{Acknowledgements}
I would like to thank Felipe Rincon, Ulrike Tillmann, and Irakli Patchkoria for valuable conversations
and encouragement at early stages of this work.  I was partially supported by EPSRC grant EP/R018472/1.

\section{Flavours of matroids}

In this section we briefly recall the definition of matroids, valuated matroids, and oriented
matroids.  We then discuss the cross-breed notion of valuate signed matroids.  We will work with
matroids in terms of their Pl\"ucker vectors.  Our perspective is based on the ideas of Baker and
Bowler \cite{Baker-Bowler-partial-hyperstructures} and Anderson and Davis
\cite{Anderson-hyperfields}.

\subsection{Alternating functions}
Fix a set $E$ and let $V$ be a set with an involution (i.e., an action of the group $\mathbb{Z}/2 =
\{+1, -1\}$) and a distinguished element $0$ fixed by the involution.  A function $f: E^d \to V$ is
\emph{alternating} if
\[f(x_1, \ldots, x_d) = \mathrm{sign}(\sigma) \cdot f(x_{\sigma(1)},  \ldots, x_{\sigma(d)})
\]
 for any permutation $\sigma \in \Sigma_d$, and $f(x_1, \ldots, x_d)=0$ if the arguments $x_i$ are
 not all distinct.

If we choose a total ordering of $E$ then any subset $X$ inherits a total ordering $(x_1,
x_2,\ldots)$. This induces a bijection between alternating functions $f: E^d \to V$ and functions
$g: \binom{E}{d} \to V$, where $\binom{E}{d}$ denotes the set of size $d$ subsets of $E$; a function
$g$ on unordered sets determines an alternating function by $f(x_1, \ldots, x_d) =
\mathrm{sign}(\sigma)\cdot g(\{x_{\sigma(1)},  \ldots, x_{\sigma(d)} \} )$ where $\sigma$ is the
unique permutation such that $x_{\sigma(1)} < \cdots < x_{\sigma(d)}$.  In what follows we will
frequently use this bijection implicitly.

\subsection{Ordinary matroids}
A \emph{matroid} of rank $d$ on a finite ground set $E$ is a nonzero mapping
\[
	p: \binom{E}{d} \to \{0,1\}
\] 
such that for any pair of sets $X \in \binom{E}{d+1}$ and $Y \in \binom{E}{d-1}$, if $p(X-i)$ and
$p(Y\cup i)$ are both 1 for some $i \in X \smallsetminus Y$, then there exists another element $j
\in X \smallsetminus Y$ (with $j\neq i$) for which $p(X-j)$ and $p(Y\cup j)$ are both 1. The
function $p$ is sometimes called a \emph{tropical Pl\"ucker vector}, and its components are the
\emph{Pl\"ucker coordinates}.  The sets $B$ for which $p(B)=1$ are called the \emph{bases} of $p$.

There is a partial order on the set of all matroids where $p \preceq q$ if $p(B) \geq q(B)$ for all
$B$. I.e., in passing from $p$ to $q$, some sets might cease to be bases, but no new bases appear.
When $p \preceq q$ we say that $q$ is a specialization of $p$. The uniform matroid, which is given
by the constant function $p=1$, is the unique minimal matroid with respect to specialization, and
the maximal matroids are those that have only a single basis.

\subsection{Valuated matroids}
Let $\mathbb{T} = \mathbb{R}\cup \{\infty\}$, topologised so that it is homeomorphic to a
half-closed interval; i.e., the map $x \mapsto e^{-x}$ is a homeomorphism to $[0,\infty)$. Given a
function $\phi: \binom{E}{d} \to \mathbb{T}$, the \emph{initial datum} of $\phi$ is the function
$I_\phi$ that sends each pair of distinct sets $(X\in \binom{E}{d+1}, Y \in \binom{E}{d-1})$ to the
set
\[
	\argmin_{i\in X\smallsetminus Y} \big( \phi(X - i) + \phi(Y\cup i) \big) \subset X \smallsetminus Y.
\]
We say that $\phi$ is a \emph{tropical Pl\"ucker vector} of rank $d$ on ground set $E$ if it is not
identically $\infty$ and $I_\phi(X,Y)$ always has cardinality at least 2.  The additive group
$(\mathbb{R},+)$ acts on the set $\{\phi \text{ a tropical Pl\"ucker vector} \}$ via
\[
	(\lambda + \phi)(B) = \phi(B) + \lambda
\]
for $\lambda \in \mathbb{R}$, and a \emph{valuated matroid} is an orbit equivalence class.

Given a tropical Pl\"ucker vector $\phi$, the underlying matroid
$\underline{\phi}$ is defined by
\[
	B \mapsto 
	\begin{cases}
		0 & p(B) = \infty\\
		1 & p(B) \neq \infty.
	\end{cases}
\]
We can think of $\phi$ as the datum of a valuation on the underlying matroid $\underline{\phi}$ that
assigns a real number to each basis. We say that a function $I$ sending each pair $(X,Y)$ to a
subset $I(X,Y) \subset X\smallsetminus Y$ is compatible with a matroid $p$ if there exists a
tropical Pl\"ucker vector $\phi$ with $I=I_\phi$ and $p = \underline{\phi}$.

Let $\widetilde{\mathit{Dr}}(d,E) \subset \mathbb{T}^{\binom{E}{d}} \smallsetminus \{\infty\}$
denote the space of all tropical Pl\"ucker vectors. There is a diagonal action of the additive group
$(\mathbb{R},+)$ on $\mathbb{T}^{\binom{E}{d}} \smallsetminus \{\infty\}$ extending the action on
the set of tropical Pl\"ucker vectors; the quotient is known as tropical projective space
$\mathbb{P}(\mathbb{T}^{\binom{E}{d}})$. It is a manifold with boundary that is diffeomorphic to a
simplex.  The image of $\widetilde{\mathit{Dr}}(d,E)$ in the tropical projective space is the space
of valuated matroids and it is known as the \emph{Dressian} and denoted $\mathit{Dr}(d,E)$.

\subsection{Oriented matroids}
A \emph{chirotope} is a map
\[
	\chi: E^d \to \{-1,0,1,\}
\] 
such that:
\begin{enumerate}
	\item $\chi$ is not identically zero;
	\item $\chi$ is alternating;
	\item Given $X = (x_1, \ldots, x_{d+1}) \in E^{d+1}$ and $Y=(y_1, \ldots, y_{d-1}) \in E^{d-1}$, as $k$ varies from 1 to $d+1$, either the elements
	\begin{equation}\label{eq:chirotope-expression}
	(-1)^k \chi(x_1, \ldots, \widehat{x_k}, \ldots, x_{d+1}) \chi(x_k, y_1, \ldots, y_{d-1})
	 \in \{-1,0,+1\}
	\end{equation}
	are all zero or this list contains both $+1$ and $-1$.
\end{enumerate}
Chirotopes are the Pl\"ucker vectors of oriented matroids.  Let $[\chi]$ denote the equivalence
class of $\chi$ under the identification $\chi \sim -\chi$.  An \emph{oriented matroid} is an
equivalence class $[\chi].$

There is a specialization partial order on the set of chirotopes on $E$ of fixed rank, where
\[
 \chi\preceq \chi' \text{ if $\chi'(B)$ is either $\chi(B)$ or $0$ for each $B\in E^d$},
\]
and this descends to a partial order on oriented matroids.  The \emph{MacPhersonian}
$\mathit{MacP}(d,E)$ is the nerve of the poset of all rank $d$ oriented matroids on $E$.

Since a chirotope $\chi$ is an alternating function, choosing an ordering of $E$ induces an ordering
of each subset $B\subset E$, and hence we may encode $\chi$ as a function $\binom{E}{d} \to
\{-1,0,1\}$. We can thus think of an oriented matroid as an underlying ordinary matroid equipped with the
additional data of a sign for each basis subject to the above conditions.

If $\chi$ is a chirotope, then the absolute value $|\chi|: E^d \to \{0,1\}$ is an ordinary matroid
which we call the \emph{underlying matroid} of $\chi$.  We say that an initial datum $I$ is
\emph{compatible} with $\chi$ if it is compatible with the matroid $|\chi|$ and for all
$(X=\{x_i\},Y=\{y_i\})\in
\binom{E}{d+1} \times \binom{E}{d-1}$, one or the other of the following two conditions holds:
\begin{enumerate}
\item The expression in \eqref{eq:chirotope-expression} is identically zero for all $k$;
\item There exist indices $k_+$ and $k_-$ such that $x_{k_+}, x_{k_-} \in I(X,Y)$ and
\eqref{eq:chirotope-expression} is  $+1$ when $k=k_+$ and $-1$ when $k=k_-$.
\end{enumerate}

\subsection{Oriented valuated matroids}\label{sec:oriented-valuated-matroids}

Oriented valuated matroids are hybrid objects combining oriented matroids and valuated matroids in a compatible way.

\begin{definition}\label{def:chirotope-compatibility}
Given a rank $d$ tropical Pl\"ucker vector $\phi$ on ground set $E$, a chirotope $\chi$ with the same ground set and rank is said to be \emph{compatible} with $\phi$ if:
\begin{enumerate}
\item $\phi(B) = \infty$ if and only if $\chi(B) = 0$;

\item For ordered sets $X=\{x_1, \ldots, x_{d+1}\} \in \binom{E}{d+1}$ and $Y =\{y_1, \ldots, y_{d-1}\}\in \binom{E}{d-1}$, let
$I \subset \{1, \ldots, d+1\}$ denote the set of indices for which
the minimum of the set of the expressions 
\[
	\phi(\{x_1, \ldots, \widehat{x_i}, \ldots, x_{d+1}\}) + \phi(\{x_i, y_1, \ldots, y_{d-1}\})
\]
it attained. Then there exist a pair of elements $i,j \in I$ such that
\[
(-1)^i \chi(x_1, \ldots, \widehat{x_i}, \ldots, x_{d+1})\chi(x_i, y_1, \ldots, y_{d-1})
\]
 and 
 \[
 (-1)^j \chi(x_1, \ldots, \widehat{x_j}, \ldots, x_{d+1})\chi(x_j,y_1, \ldots, y_{d-1})
 \]
 have opposite signs.
\end{enumerate}
An \emph{oriented tropical Pl\"ucker vector} is a compatible pair $(\phi, \chi)$. 
\end{definition}

We can package a compatible pair $(\phi,\chi)$ as a single map $\Phi: E^d
\to \mathbb{R}$, where 
\[\phi = -\log |\Phi| \text{ and } \chi = \mathrm{sign}(\Phi).
\]
The condition that $\Phi$ is an oriented tropical Pl\"ucker vector becomes: for each $X=(x_1,
\ldots, x_{d+1}) \in E^{d+1}$ and $Y=(y_1, \ldots, y_{d-1}) \in E^{d-1}$, either the numbers
\[
	\{(-1)^i \Phi(x_1, \ldots, \widehat{x_i}, \ldots, x_{d+1}) \Phi(x_i, y_1, \ldots, y_{d-1})\}_{i=1\ldots d+1}
\]
are all zero, or the maximum modulus occurs with both signs. The group $\mathbb{R}^\times$ acts on
the set of such $\Phi$ by multiplication, and an \emph{oriented valuated matroid} is an orbit.

By choosing an ordering of $E$, we can represent $\Phi$ as a map $\binom{E}{d} \to \mathbb{R}$;
expressing the signs in this format is less straightforward, but it allows us embed the set of
oriented tropical Pl\"ucker vectors into the Euclidean space $\mathbb{R}^{\binom{E}{d}}$ and make a
direct comparison to the space of tropical Pl\"ucker vectors, which we do in the following section.
Let $\widetilde{\Mac}(d,E) \subset
\mathbb{R}^{\binom{E}{d}} \smallsetminus \{0\}$ denote the space of all oriented tropical Pl\"ucker
vectors, topologized as a subspace of Euclidean space, and let
\[ 
\Mac(d,E) 
\subset (\mathbb{R}^{\binom{E}{d}} \smallsetminus \{0\} )/ \mathbb{R}^\times = \mathbb{P}(\mathbb{R}^{\binom{E}{d}})
\] 
denote the set of all oriented valuated matroids.

\begin{remark}
In the terminology of \cite{Anderson-hyperfields}, this is the Grassmannian of matroids over the
`real tropical' hyperfield. Note that Anderson and Davis work with a different topology on the real
line which induces a different topology of the space of oriented valuated matroids.  By
\cite[Theorem 1.1]{Anderson-hyperfields}, their topology leads to a weak equivalence
$|\mathit{MacP}(d,E)|
\stackrel{\simeq}{\to} \Mac(d,E)$.  We will show that the Euclidean topology in fact leads to a CW structure on $\Mac(d,E)$ and hence a homotopy equivalence.
\end{remark}

An inclusion $E_1 \hookrightarrow E_2$ of finite sets induces an injective map $\Mac(d,E_1)
\hookrightarrow \Mac(d,E_2)$.  When $E$ is countably infinite, we define $\Mac(d,E)$ as the colimit
of $\Mac(d,F)$ over finite subsets $F \subset E$.

\section{Polyhedral structures}

Our goal here is to describe a polyhedral structure on the space $\Mac(d,E)$ of oriented
valuated matroids.  We define this structure by lifting the polyhedral fan structure on the
Dressian, which has been studied by various authors such as
\cite{Herrmann1,Herrmann2,Olarte-local-dressians,Brandt-dressians}. In fact, there are two natural
ways to define a fan structure on the Dressian (the Pl\"ucker fan and the secondary fan), and they
were shown to coincide in \cite{Olarte-local-dressians}. In its Pl\"ucker fan description, the cones
are determined by the initial forms of the quadratic Pl\"ucker expression.  We will review this
structure and then describe the corresponding structure on the space $\Mac(d,E)$ of oriented
valuated matroids.

\subsection{Polyhedral structure of the Dressian}

Here we present a description of the Pl\"ucker fan structure of the Dressian $\mathit{Dr}(d,E)$.
Given a matroid $p$ and a compatible initial datum $I$, let $C(p,I) \subset
\mathbb{T}^{\binom{E}{d}} \smallsetminus \{\infty\}$ denote the space of all tropical Pl\"ucker
vectors having underlying matroid $p$ and initial datum $I$.  The sets $C(p,I)$ are invariant under
the diagonal action of $(\mathbb{R},+)$ and so descend to subsets of the tropical projective space
$\mathbb{P}(\mathbb{T}^{\binom{E}{d}})$.  The space $\widetilde{\mathit{Dr}}(d,E)$ of all tropical
Pl\"ucker vectors is clearly partitioned as the set-theoretic disjoint union over all $C(p,I)$, and
this partition induces a partition of $\mathit{Dr}(d,E)$.

The ambient tropical affine space $\mathbb{T}^{\binom{E}{d}}$ is stratified by the collections of
coordinates that are $\infty$.  Each stratum is canonically a Euclidean space $\mathbb{R}^n$.  If
$\mathrm{supp}(p) \subset \binom{E}{d}$ denotes the support of $p$ (i.e., the set of bases of the
matroid $p$), then $C(p,I)$ is contained in the stratum $\mathbb{R}^{\mathrm{supp}(p)}$.

\begin{proposition}
The space $C(p,I) \subset \mathbb{R}^{\mathrm{supp}(p)}$ is the relative interior of a convex
polyhedral cone that is invariant under translation along the vector $(1, \dots, 1)$.
The boundary of the closure $\overline{C(p,I)}$ in $\mathbb{T}^{\binom{E}{d}}$ is the union of
those cones $C(p',I')$ for which $p \preceq p'$ and $I(X,Y)
\subset I'(X,Y)$ for all $X$ and $Y$.
\end{proposition}
\begin{proof}
The set $C(p,I) \subset \mathbb{R}^{\mathrm{supp}(p)}$ is cut out by the following set of linear equations and inequalities:
\begin{enumerate}
\item For $X\in \binom{E}{d+1}, Y\in \binom{E}{d-1}$, $i \in I(X,Y)$ and $j \notin I(X,Y)$ we have
\[\phi(X-i) + \phi(Y+i) < \phi(X-j) + \phi(Y+j).\]
\item For $X\in \binom{E}{d+1}, Y\in \binom{E}{d-1}$, $i,j \in I(X,Y)$ we have
\[\phi(X-i) + \phi(Y+i) = \phi(X-j) + \phi(Y+j).\]
\end{enumerate}

Suppose we have a sequence of points $\phi_i$ in $C(p,I)$ converging to a point $\phi_\infty$ in the
boundary.  This limit point will lie in some set $C(p',I')$, and we now derive constraints on
$(p',I')$. It is clear that $\phi_\infty$ will satisfy all conditions of type (2) above, and so
there are only two things that can happen:
\begin{enumerate}
\item  In the limit some of the strict inequalities of condition (1) can become equalities.
This means that there are one or more pairs $(X,Y)$ such that $I(X,Y) \subsetneq I'(X,Y)$.  

\item Some of the values of $\phi_i(B)$ can increase to $\infty$, so $p \preceq
\underline{\phi_\infty}$. This can only happen if $B$ does not appear in any of the type (2)
equalities for $(p,I)$; i.e., $B \neq X-i$ or $Y+i$ for any $X,Y$ and $i\in I(X,Y)$.
\end{enumerate}
This shows that the boundary is contained in the union of cells $C(p',I')$ for which $p' \succeq p$
and $I'(X,Y) \supset I(X,Y)$ for all $X,Y$.   Conversely, we must show that any point $\phi$ in a
cell $C(p',I')$ satisfying this condition can be perturbed to a point in $C(p,I)$.  We do this in
two steps.

If $B$ is a basis for $p$ but not $p'$, then $\phi(B)=\infty$, and $\phi(B)$ does not appear in any
of the type (2) equalities. However, $\phi(B)$ can appear on the right hand side of some of the type
(1) inequalities, and setting $\phi'(B)$ to be a sufficiently large number will still satisfy these
same inequalities.  Thus we obtain a perturbation of $\phi$ to a point $\phi' \in  C(p,I') \subset
\mathbb{R}^{\mathrm{supp}(p)}$. The closure of $C(p,I)$ in the stratum
$\mathbb{R}^{\mathrm{supp}(p)}$ is
\[
	\bigcup_{I' \supset I} C(p,I').
\]
We can thus perturb $\phi'$ to a point $\phi''$ in $C(p,I)$, as desired.
\end{proof}

\begin{corollary}
The space $\widetilde{\mathit{Dr}}(d,E)$ is a polyhedral fan with cones $C(p,I)$, and this descends
to a polyhedral fan structure on the Dressian $\mathit{Dr}(d,E)$.
\end{corollary}

\subsection{Polyhedral structure for the space of oriented valuated matroids}

Consider the map 
\[
	\mathbb{R}^{\binom{E}{d}} \to \mathbb{T}^{\binom{E}{d}}
\]
given by applying $x \mapsto -\log |x|$ component-wise.  This induces a map of subspaces $\widetilde{\Mac}(d,E) \to \widetilde{\mathit{Dr}}(d,E)$ which projectivises to a map
\[
\Mac(d,E) \to \mathit{Dr}(d,E).
\]
This map is given by forgetting the orientation data and sending an oriented valuated matroid to its
underlying valuated matroid.   We will show that the polyhedral fan structure on the Dressian lifts
to a CW complex structure on $\Mac(d,E)$ (that can be viewed as a polyhedral complex structure in appropriate coordinates).

\begin{lemma}\label{lem:chirotope-compatible-across-cells}
If a chirotope $\chi$ is compatible with a tropical Pl\"ucker vector $\phi \in C(p,I)$, then it is also compatible with any other tropical Pl\"ucker vector $\phi' \in C(p,I)$.
\end{lemma}
\begin{proof}
This follows immediately from the definition of the compatibility condition of Definition
\ref{def:chirotope-compatibility}, since condition (1) of compatibility of $\chi$ with $\phi$ only
depends on the underlying matroid $\underline{\phi} = p$ and condition (2) only depends on the initial
datum $I$.
\end{proof}

Let $\chi$ be a chirotope and $I$ a compatible initial datum.  Let $\widetilde{D}(\chi,I)$ denote the set of all oriented valuated tropical Pl\"ucker vectors $\Phi$ such that
\begin{enumerate}
	\item The valuated matroid $|\Phi|$ has initial datum $I_{|\Phi|} = I$; 
	\item $\mathrm{sign}(\Phi) = \chi$.
\end{enumerate}
Let $D([\chi],I)$ denote the image of $\widetilde{D}(\chi,I)$ in the space $\Mac(d,E) \subset
\mathbb{P}(\mathbb{R}^{\binom{E}{d}})$; this is the set of all oriented valuated matroids $[\Phi]$
with initial datum $I$ and underlying oriented matroid $[\chi]$.  It is clear that the preimage of $D([\chi],I)$ in $\widetilde{\Mac}(d,E)$ is the disjoint union of two cells $\widetilde{D}(\chi,I) \sqcup \widetilde{D}(-\chi,I)$.

\begin{theorem}\label{thm:cw-structure}
The space $\widetilde{\Mac}(d,E)$ is a CW complex with cells $\widetilde{D}(\chi,I)$, and
$\Mac(d,E)$ is a CW complex with cells $D([\chi],I)$.  Moreover, under the map
$\mathbb{R}^{\binom{E}{d}} \to \mathbb{T}^{\binom{E}{d}}$, each open cell $\widetilde{D}(\chi,I)$
maps homeomorphically onto the relative interior of a convex polyhedral cone.
\end{theorem}
\begin{proof}
Clearly $\widetilde{\Mac}(d,E)$ is the set-theoretic disjoint union of the sets $D(\chi, I)$.
If $D(\chi, I)$ is nonempty, then by Lemma \ref{lem:chirotope-compatible-across-cells} it is homeomorphic to the cell $C(p,I)$. Since the boundary of each cell $C(p,I) \subset \widetilde{Dr}(d,E)$ is a union of lower dimensional cells, the same is true for $D(\chi, I)$.  The corresponding claims for $\Mac(d,E)$ follow immediately.
\end{proof}

\begin{theorem}
There is a homotopy equivalence $|\mathit{MacP}(d,E)| \simeq \Mac(d,E)$.
\end{theorem}
\begin{proof}
Consider the open covering $\mathscr{U}$ of $\Mac(d,E)$ by open stars of cells.  As a consequence of
Theorem \ref{thm:cw-structure},  this is a good cover, and by the Nerve Theorem
(\cite{Borsuk-nerve-theorem}, \cite[p. 141]{Weil-nerve-theorem}, or \cite[Theorem
2]{McCord-nerve-theorem}), $\Mac(d,E)$ is homotopy equivalent to the nerve of the covering
$\mathscr{U}$.  This nerve is isomorphic to the nerve of the poset $\mathcal{C}$ of cells, and by Theorem \ref{thm:cw-structure}, $\mathcal{C}$ is the set of pairs $([\chi],I)$ with $([\chi],I)
\leq ([\chi'],I')$ if $I(X,Y) \subset I'(X,Y)$ for any $X,Y$, and and $[\chi] \leq [\chi']$.

Sending $([\chi],I)$ to $[\chi]$ defines a morphism of posets 
\[
\pi: \mathcal{C} \to \mathit{MacP}(d,E).
\]
We will show that $\pi$ gives a homotopy equivalence on nerves by first showing that each geometric fibre is contractible and then showing that the inclusion of each geometric fibre into the corresponding homotopy fibre is a homotopy equivalence.

Let $\chi$ be a chirotope, which gives a vertex of $\mathit{MacP}(d,E)$. The geometric fibre
$\mathcal{F}_{[\chi]}$ of $\pi$ over $[\chi]$ is the sub-poset of $\mathcal{C}$ consisting of all pairs
$([\chi], I)$.  This poset has a final object given by $([\chi],I_{\mathrm{max}}^{\chi})$,
where $I_{\mathrm{max}}^{\chi}$ is the maximal initial datum compatible with $\chi$:
\[
	x\in I(X,Y) \text{ if } |\chi|(X - x)\cdot |\chi|(Y + x) \neq 0.
\]
Hence the nerve of $\mathcal{F}_{[\chi]}$ is contractible.  

Now consider the fibre category $h\mathcal{F}_{[\chi]} = [\chi] \backslash \pi$ consisting of pairs
$([\tau],I)$ such that $[\tau] \geq [\chi]$. We think of this as a model for the homotopy fibre of
$\pi$.  Let $j: \mathcal{F}_{[\chi]} \hookrightarrow h\mathcal{F}_{[\chi]}$ denote the inclusion.  We
will show that $j$ is an equivalence by showing that its fibre categories are all contractible.  Given an object $([\tau],I)$ of $h\mathcal{F}_{[\chi]}$, the fibre category $j / ([\tau],I)$ is the poset of initial data $J$ compatible with $\chi$ and such that $J(X,Y) \subset I(X,Y)$.  This fibre category has $I$ as a final object since an initial datum compatible with $\tau$ is compatible with any chirotope $\tau' \leq \tau$, and so in particular $I$ is compatible with $\chi$ because $\chi \leq \tau$

Now, by Quillen's Theorem A (\cite{Quillen-Thm-A} or \cite[Theorem 6]{McCord-Thm-A}), the map $j$
induces a homotopy equivalence of geometric realizations of nerves,
\[
	|\mathcal{F}_{[\chi]}| \hookrightarrow |h\mathcal{F}_{[\chi]}|, 
\]
and hence $|\mathcal{F}_{[\chi]}|$ is contractible for any choice of $\chi$.  Then, by Quillen's Theorem A once again, $\pi$ induces a homotopy equivalence $|\mathcal{C}| \to |\mathit{MacP}(d,E)|$.
\end{proof}

\section{The $H$-space structure}

Here we investigate maps between the spaces $\Mac(d,E)$ corresponding to direct sums.

An injective map $\alpha: E\to F$ induces an injective map $\alpha_*: \Mac(d,E) \hookrightarrow
\Mac(d,F)$. Given two injective maps $\alpha,\beta: E \to F$, in this section we will construct a
homotopy between $\alpha_*$ and $\beta_*$.  Moreover, we will show that these homotopies are related
by higher homotopies.

\subsection{Direct sums}

Suppose $p_1$ and and $p_2$ are matroids of rank $d_1$ and $d_2$ on ground sets $E_1$ and $E_2$.
The direct sum $p_1 \oplus p_2$ is a matroid of rank $d_1 + d_2$ on $E_1 \sqcup E_2$ whose bases are
those sets of the form $B_1 \sqcup B_2$ where each $B_i$ is a basis for $p_i$.  This operation lifts
to oriented valuated matroids as follows.  An alternating function $(E_1 \sqcup E_2)^{d_1 + d_2} \to
\mathbb{R}$ is uniquely determined by its restriction to the subset $E_1^{d_1} \times E_2^{d_2}$,
and here we define $\Phi_1 \oplus \Phi_2$ to simply be the product of $\Phi_1$ on the first factor
times $\Phi_2$ on the second.

This direct sum operation defines a continuous map 
\[
	\Mac(d_1, E_1) \times \Mac(d_2, E_2) \to \Mac(d_1 + d_2, E_1 \sqcup E_2).
\]

If $E_1 = E_2 = \mathbb{N}$, then we may choose an injective map $\alpha: \mathbb{N} \sqcup
\mathbb{N} \to \mathbb{N}$.  Composing with $\alpha_*$ yields a map
\[
	\Mac(d_1, \mathbb{N}) \times \Mac(d_2, \mathbb{N}) \to \Mac(d_1 + d_2, \mathbb{N}).
\]
Let $\mathcal{M} = \coprod_{d=0}^\infty \Mac(d,\mathbb{N})$.  We then have a binary operation
\[
	\mu: \mathcal{M} \times \mathcal{M} \to \mathcal{M}.
\]
The space $\Mac(0,\mathbb{N})$ is a single point (corresponding to the unique rank zero oriented
valuated matroid), and this point is a unit for $\mu$.  Note however that $\mu$ is neither
commutative nor associative.  Nevertheless, below we show that it is associative and commutative up
to homotopy so long as the image of $\alpha$ has infinite complement.  Our main result will then be
to show that it extends to an $E_\infty$ structure.

\subsection{Matroid sliding}

Given a tuple of injective maps $A=(\alpha_k: E \hookrightarrow F)_{k=0}^n$ that are pairwise
disjoint, let $x\mapsto \overline{x}$ be the map
\begin{equation}\label{eq:pi-map}
\bigcup_k \alpha_k(E) \to E
\end{equation}
that restricts to $\alpha_k^{-1}$ on the image of $\alpha_k$ (for each $k$).  For $(x_1,
\ldots, x_d) \in F^d$, let $b_i\in \mathbb{N}$ ($i=1 \ldots, d$) denote the number of components in
$\alpha_i(E)$. Given an oriented tropical Pl\"ucker vector $\Phi$ of rank $d$ on $E$ and a point
$t=(t_0, \ldots, t_n) \in \Delta^{n}$ (i.e., $t_k \in [0,1]$ and $\sum_k t_k = 1$), consider the
function $\Phi_t^A : F^d \to \mathbb{R}$ given by the formula
\[
\Phi_t^A(x_1, \ldots, x_d) = 
\begin{cases} 
\Phi(\overline{x_1}, \ldots, \overline{x_d}) t_0^{b_0} t_1^{b_1} \cdots t_n^{b_n} & \sum b_k = d\\
0 &  \text{otherwise.}
\end{cases}
\]

\begin{proposition}\label{prop:matroid-sliding}
If $\Phi$ is an oriented tropical Pl\"ucker vector on ground set $E$, then the function $\Phi_t^A$ defined
above is an oriented tropical Pl\"ucker vector on $F$,   Moreover, when $(t_0, \ldots, t_n)$ is at
the $k^{th}$ vertex of $\Delta^n$ (i.e., $t_k = 1$ and $t_\ell = 0$ for $\ell \neq k$), then $\Phi_t^A = (\alpha_k)_* \Phi$.
\end{proposition}
\begin{proof}
For any tuples $X=(x_1,\ldots, x_{d+1}) \in F^{d+1}$ and $Y=(y_1, \ldots, y_{d-1}) \in F^{d-1}$, consider the list of real numbers
\begin{equation}\label{eq:plucker}
	 \{\Phi_t^A(x_1, \ldots, \widehat{x_i}, \ldots, x_{d+1}) \cdot \Phi_t^A(x_i, y_1, \ldots, y_{d-1}) \in \mathbb{R}\}_{i=1\ldots, d+1}.
\end{equation}
We will proceed by considering the various ways in which the tuples $X$ and $Y$ can meet the fibres of
the map \eqref{eq:pi-map}. Let $a_k$ and $b_k$ denote the number of components of $X$ and $Y$
respectively lying in the image of $\alpha_k$.

If $X$ contains 3 or more elements in the same fibre, then for any $i$, the tuple \[(\overline{x_1},
\dots, \widehat{\overline{x_i}}, \dots, \overline{x_n})\] will have at least two components that are
equal, and so $\Phi_t^A(x_1, \ldots, \widehat{x_i}, \ldots, x_{d+1}) = 0$.  Thus all of the numbers
in \eqref{eq:plucker} are zero, and so the oriented tropical Pl\"ucker relation at $(X,Y)$ is
trivially satisfied.

Suppose $X$ contains a pair of elements $x_j,x_{j'}$ within the same fibre, so $\overline{x_j} =
\overline{x_{j'}}$, with  $x_j \in \mathrm{Im}(\alpha_k)$ and $x_{j'} \in \mathrm{Im}(\alpha_{k'})$.
Then $\Phi_t^A(x_1, \ldots, \widehat{x_i}, \ldots, x_{d+1})$ can only be nonzero if either $i = j$
or $i=j'$.  Thus there are at most two nonzero terms in the \eqref{eq:plucker}. We have
\begin{multline*}
(-1)^j \Phi_t^A(x_1, \ldots, \widehat{x_j}, \ldots, x_{d+1}) \cdot \Phi_t^A(x_j, y_1, \ldots, y_{d-1}) \\
=(-1)^j \Phi(\overline{x_1}, \ldots, \widehat{\overline{x_j}}, \ldots, \overline{x_{d+1}}) t_0^{a_0} \cdots t_k^{a_k - 1} \cdots t_n^{a_n} \cdot \Phi(\overline{x_j}, \overline{y_1}, \ldots, \overline{y_{d-1}})t_0^{b_0} \cdots t_k^{b_k + 1} \cdots t_n^{b_n}\\
= (-1)^j \Phi(\overline{x_1}, \ldots, \widehat{\overline{x_j}}, \ldots, \overline{x_{d+1}}) \cdot \Phi(\overline{x_j}, \overline{y_1}, \ldots, \overline{y_{d-1}})t_0^{a_0 + b_0} \cdots t_n^{a_n + b_n},
\end{multline*}
and similarly
\begin{multline*}
(-1)^{j'} \Phi_t^A(x_1, \ldots, \widehat{x_{j'}}, \ldots, x_{d+1}) \cdot \Phi_t^A(x_{j'}, y_1, \ldots, y_{d-1}) \\
= (-1)^{j'} \Phi(\overline{x_1}, \ldots, \widehat{\overline{x_{j'}}}, \ldots, \overline{x_{d+1}}) \cdot \Phi(\overline{x_{j'}}, \overline{y_1}, \ldots, \overline{y_{d-1}})t_0^{a_0 + b_0} \cdots t_n^{a_n + b_n},
\end{multline*}
Since $\overline{x_j} = \overline{x_{j'}}$ and $\Phi$ is alternating, these two expressions are equal up to a sign.  This sign is $(-1)^{j+j'}$ times the sign of the unique permutation that identifies $(\overline{x_1}, \ldots, \widehat{\overline{x_j}}, \ldots, \overline{x_{d+1}})$ with $(\overline{x_1}, \ldots, \widehat{\overline{x_{j'}}}, \ldots, \overline{x_{d+1}})$, which is $(-1)^{j+j' + 1}$.

Now consider the case that $X$ contains at most a single element in each fibre. We then have
\begin{multline*}
(-1)^i \Phi_t^A(x_1, \ldots, \widehat{x_i}, \ldots, x_{d+1}) \cdot \Phi_t^A(x_i, y_1, \ldots, y_{d-1}) \\
= (-1)^i \Phi(\overline{x_1}, \ldots, \widehat{\overline{x_i}}, \ldots, \overline{x_{d+1}}) \cdot \Phi(\overline{x_i}, \overline{y_1}, \ldots, \overline{y_{d-1}})t_0^{a_0 + b_0} \cdots t_n^{a_n + b_n}.
\end{multline*}
The monomial in the $t$s is a constant independent of $i$, and so the oriented tropical Pl\"ucker relation at $(X,Y)$ is satisfied by $\Phi_t^A$ because it is satisfied by $\Phi$ at $(\overline{X},\overline{Y})$.
\end{proof}

As an immediate consequence, if $\alpha$ and $\beta$ are injective maps $E \to F$, then the induced maps $\alpha_*, \beta_*: \Mac(d,E) \to \Mac(d,F)$ are homotopic.  Moreover, in the special case of a countably infinite ground set, we have:

\begin{corollary}
If $\alpha, \beta: E \to \mathbb{N}$ are injective maps each with infinite complement, then $\alpha_*$ and
$\beta_*$ are homotopic.
\end{corollary}
\begin{proof}
If $\alpha(E)^c \cap \beta(E)^c$ is infinite, then choose an injective map $\gamma: E \to \alpha(E)^c \cap \beta(E)^c \subset \mathbb{N}$.  By Proposition \ref{prop:matroid-sliding}, both $\alpha_*$ and $\beta_*$ are homotopic to $\gamma_*$.  If the intersection $\alpha(E)^c \cap \beta(E)^c$ is finite, then $\alpha(E) \smallsetminus \beta(E)$ and $\beta(E) \smallsetminus \alpha(E)$ must both be infinite, so we can choose injective maps $\gamma: E \to \alpha(E) \smallsetminus \beta(E)$ and $\delta: E \to \beta(E) \smallsetminus \alpha(E)$.  We then have a sequence of homotopies
\[
	\beta_* \simeq \gamma_* \simeq \delta_* \simeq \alpha.\qedhere
\]
\end{proof}

\section{The simplicial operad of injective maps}

In this section we will construct an operad that acts on the matroid Grassmannians via the slide
moves defined in the previous section, and then we will show that it is in fact an $E_\infty$
operad.  

One of the most familiar examples of an $E_\infty$ operad is the litle discs operad in which the
space of $n$-ary operations is colimit as $d\to \infty$ of the space of configurations of $n$
disjoint $d$-discs in a large disc.  The operad we construct is analogous to this, but instead of little
discs in a big disc, we use infinite subsets of $\mathbb{N}$.

\subsection{Background on operads}

We recall some definitions and results from \cite{May-infinite-loop-spaces,May-group-completion}. An
operad consists of:
\begin{enumerate}
\item A functor $\mathscr{O}$ from the category of finite sets and bijections to
spaces (so the symmetric groupp $\Sigma_A$ acts on each space $\mathscr{O}(A))$.
\item For each map of finite sets $\gamma: B \to A$, a composition map
\[
	\gamma^*: \mathscr{O}(A) \times  \prod_{a \in A} \mathscr{O}(\gamma^{-1}(a)) \to \mathscr{O}(B) 
\]
subject to the conditions below.
\end{enumerate}
The composition maps are required to be associative, unital, and equivariant in the following sense:
\begin{enumerate}
\item (Unital) There is a distinguished element $1 \in \mathscr{O}(\{*\})$ such that for any $x\in
\mathscr{O}(A)$ the composition corresponding to the identity map $\mathrm{id}: A \to A$ satisfies
$\mathrm{id}^*(x,1\ldots, 1) = x$ , and the composition corresponding to the map $\pi: A \to \{*\}$
satisfies $\pi^*(1, x) = x$.

\item (Equivariant) Given $\gamma:B \to A$ and a permutation $\sigma \in \Sigma_B$ that preserves
the relation of elements being in the same fibre over $A$, there are induced automorphisms of
$\mathscr{O}(A)$, $\prod_{a \in A} \mathscr{O}(\gamma^{-1}(a))$, and $\mathscr{O}(B)$, and the
diagram
\[
\begin{tikzcd}
\mathscr{O}(A) \times  \prod_{a \in A} \mathscr{O}(\gamma^{-1}(a))
 \arrow[r,"\gamma^*"] 
 \arrow[d,"\sigma_*"] & \mathscr{O}(B) \arrow[d,"\sigma_*"] \\
\mathscr{O}(A) \times  \prod_{a \in A} \mathscr{O}(\gamma^{-1}(a)) 
\arrow[r,"\gamma^*"]           & \mathscr{O}(B)         
\end{tikzcd}
\]
commutes.

\item (Associative) Given maps $C \stackrel{\tau}{\to} B \stackrel{\gamma}{\to} A$, if we let
$\tau_a$ denote the restriction of $\tau$ to $\gamma^{-1}(a)$, then
\[
	\gamma^* \circ \left(\prod_{a\in A} \tau_a^* \right) = \tau^* \circ \gamma^*.
\]
\end{enumerate}

An action of an operad $\mathscr{O}$ on a space $X$ consists of maps
\[
	\mu_A: \mathscr{O}(A)\times X^A \to X
\]
that are symmetric group equivariant, compatible with the composition maps of $\mathscr{O}$ in the sense that the diagram
\[
\begin{tikzcd}
	\mathscr{O}(A) \times \prod_{a \in A} \mathscr{O}(\gamma^{-1}(a)) \times X^B
	 \arrow[r,"\gamma^* \times \mathrm{id}"] \arrow[d,"\mathrm{id} \times \prod \mu_{\gamma^{-1}(a)}"] & 
	\mathscr{O}(B) \times X^B \arrow[d,"\mu_B"] \\
	 \mathscr{O}(A) \times X^A \arrow[r,"\mu_A"] & X
\end{tikzcd}
\]
commutes, and such that $\mu_{\{*\}}(1,x) = x$ for any $x \in X$.

Working in the category of compactly generated spaces, an \emph{$E_\infty$-operad} is an operad such
that each space $\mathscr{O}(A)$ is contractible and its action of the symmetric group $\Sigma_A$ is
free.  If an $E_\infty$-operad acts on a space $X$, then we may choose a point in the space
$\mathscr{O}(\{1,2\})$ of binary operations and so $X$ becomes a homotopy commutative $H$-space.
The recognition principle \cite[Theorem 2.3]{May-group-completion} asserts that if $X$ is group-like
(meaning that $\pi_0(X)$ is a group), then $X$ is weakly equivalent to an infinite loop space, and
if $X$ is not group like, then its group completion is weakly equivalent to an infinite loop space.

\subsection{The space of injective maps}

Fix a partition of $\mathbb{N}$ into countably many infinite cardinality pieces, 
\[
\mathbb{N} = P_0 \cup P_1 \cup \cdots.
\] 
For example, one can take $P_i$ ($i\geq 1$) to be the set of all powers of the $i^\mathit{th}$
prime, and $P_0 = (P_1 \cup P_2 \cup \cdots)^c$.  The particular choice is not important, as
different choices will yield isomorphic results. Given a finite or countably infinite set $A$, let
$\Inj(A)$ denote the simplicial complex where vertices are injective maps $\alpha: A \hookrightarrow
\mathbb{N}$ such that the image of $\alpha$ is contained in a single piece $P_i$, and $\{\alpha_0,
\ldots, \alpha_n\}$ span an $n$-simplex if their images are disjoint.  An arbitrary point of
$\Inj(A)$ can be represented as a formal finite convex sum of vertices,
\[
	\sum_{i=0}^n t_i \alpha_i
\]
where $t_i \in (0,1)$ subject to $\sum t_i = 1$, and the $\alpha_i$ span an $n$-simplex.

\begin{lemma}\label{lem:contractible}
The space $\Inj(A)$ is contractible.
\end{lemma}
\begin{proof}
Since $\Inj(A)$ is a CW complex, it suffices to show that all homotopy groups are trivial.  Since
the sphere $S^n$ is compact, the image of any map $f: S^n \to \Inj(A)$ is contained in a subcomplex
$K$ spanned by a finite set $\{\alpha_1, \ldots, \alpha_N\}$ of vertices.  Each $\alpha_i$ has image
contained in some $P_{j_i}$.  Choose $M$ large enough so that $M > j_i$ for $i = 1,\ldots, N$, and
choose an injective map $\beta: A \to P_M$. Then the image of $\beta$ is disjoint from any of the
$\alpha_i$ appearing as vertices of $K$.  Thus if $\sigma$ is a simplex of $K$, then $\sigma\cup
\{\beta\}$ is also a simplex of $K$.  Hence the cone on $K$ includes into $\Inj(A)$ by sending the
cone apex to the vertex $\beta$. It follows that 
\[
f: S^n \to K \hookrightarrow \mathrm{Cone}(K) \hookrightarrow \Inj(A)
\]
is null-homotopic. 
\end{proof}

\subsection{An E-infinity operad}
Now we construct an operad from the spaces $\Inj(A)$. For a finite set $A$, define
\[
\Operad(A) = \begin{cases}
\Inj(A \times \mathbb{N}) & |A| \geq 2,\\
\{\text{the bijection } \{a\} \times \mathbb{N} \to \mathbb{N} \text{ given by } (a,n) \mapsto n\} & A = \{a\}.
\end{cases}
\]
 Given a map of finite sets $\gamma: B\to A$, the operad composition map
\begin{equation}\label{eq:operad-composition}
	\gamma^*: \Operad(A) \times \prod_{a \in A} \Operad(\gamma^{-1}(a))  \to \Operad(B)
\end{equation}
is constructed as follows. 

We first describe \eqref{eq:operad-composition} at the level of vertices. Let $\alpha: A \times
\mathbb{N} \to \mathbb{N}$ be a vertex of $\Operad(A)$, and for each $a\in A$ let $\beta_a:
\gamma^{-1}(a) \times \mathbb{N} \to \mathbb{N}$ be a vertex of $\Operad(\gamma^{-1}(a))$. 
We can think of $\alpha$  as a collection of injective maps $\alpha_a: \mathbb{N}
\to \mathbb{N}$ for $a\in A$ (such that if $|A|\geq 2$ then their images are all disjoint and lie
in the same piece $P_i$ of the partition of $\mathbb{N}$).   The collection $\{\beta_a\}_{a \in A}$
can be represented as a family of injective maps $\{\beta_b: \mathbb{N} \to \mathbb{N} \}_{b \in B}$
such that if $\gamma(b) = \gamma(b')$ then the images of $\beta_b$ and $\beta_{b'}$ are disjoint but
contained in the same $P_i$.  The operad composition map \eqref{eq:operad-composition} then sends
\[
	\big((\alpha_a)_{a\in A}, (\beta_b)_{b \in B}\big) \mapsto (\alpha_{\gamma(b)} \circ \beta_b)_{b\in B},
\]
which represents a vertex of the space $\Operad(B)$.

Having defined the composition map $\gamma^*$ at the level of vertices, we now extend it linearly to
products of higher dimensional simplices.  Note that if $\alpha^0, \alpha^1, \dots, \alpha^n$  are vertices of
$\Operad(A)$ that have disjoint images, then for any $b\in B$ the compositions 
\[
\left( \alpha^i_{\gamma(b)} \circ \beta_b \right)_{i=0,\dots,n} 
\]
have disjoint images.  Likewise, if $(\beta^i_a)_{i=0,\ldots, n}$ are vertices of
$\Operad(\gamma^{-1}(a))$ that span an $n$-simplex, then for any $b\in B$ the corresponding
compositions $\left( \alpha_{\gamma(b)} \circ \beta^i_b \right)_{i=1,\ldots, n}$  have disjoint
images.  It follows that $\gamma^*$ extends linearly from the 0-skeleton of $\Operad(A) \times
\prod_{a \in A} \Operad(\gamma^{-1}(a))$ to all of it.

\begin{proposition}
The spaces $\Operad(A)$ form an $E_\infty$ operad.
\end{proposition}

\begin{proof}
It is entirely straightforward to check that the operad composition maps constructed above are
associative, unital, and equivariant.  It is immediate from the definition that the symmetric group
actions are free, and contractibility of $\Operad(A)$ was proved in Lemma \ref{lem:contractible}.
\end{proof}

\subsection{The operad action on the space of oriented valuated matroids}

\begin{theorem}
The operad $\Operad$ acts on the space $\Mac = \coprod_d \Mac(d,\mathbb{N})$. 
\end{theorem}
\begin{proof}
We must give action maps  
\[
	\mu_A: \Operad(A) \times \prod_{a \in A} \Mac(d_a,\mathbb{N}) \to \Mac\left(d, \mathbb{N}\right), 	
\] 
where $d = \sum_{a \in A } d_a$. We construct these maps in two steps. First we apply the matroid
direct sum map
\[
	\prod_{a \in A} \Mac(d_a,\mathbb{N}) \to \Mac\left(d, \coprod_{a \in A} \mathbb{N}\right) = \Mac\left(d, A \times \mathbb{N}\right).
\]
Then we must define a map
\[
	\Operad(A) \times \Mac\left(d, A \times \mathbb{N}\right) \to \Mac\left(d, \mathbb{N}\right).
\]
A point of the left hand side consists of:
\begin{itemize}
\item a formal linear combination $\sum_i t_i \alpha_i$ where $t_i \in (0,1)$ are numbers such that $\sum_i t_i = 1$, and the $\alpha_i$ are disjoint injective maps $A\times \mathbb{N} \hookrightarrow \mathbb{N}$,
\item an oriented valuated matroid $\Phi$ on ground set $A\times \mathbb{N}$.
\end{itemize}
We send the pair $(\sum_i t_i \alpha_i, \Phi)$ to the oriented valuated matroid $\Phi^{(\alpha_0,
\ldots, \alpha_n)}_{(t_0, \ldots, t_n)}$ defined via matroid sliding (Proposition
\ref{prop:matroid-sliding}).
\end{proof}

The group of connected components of $\Mac$ is the additive monoid $\mathbb{N}$, and so we have that the group completion $\Omega B \Mac$ is weakly equivalent to an infinite loop space with group of connected components $\mathbb{Z}$.

\bibliographystyle{amsalpha}
\bibliography{bibliography}

\end{document}